\documentclass{article}

\usepackage{amssymb}
\usepackage{amsfonts}
\usepackage{amssymb}
\usepackage{amsmath}
\usepackage{amsthm}
\usepackage{accents}

\usepackage[normalem]{ulem}
\usepackage{soul}
\usepackage{caption}
\usepackage{enumerate}
\usepackage{tabularx}
\usepackage{centernot}
\usepackage{mathtools}
\usepackage{stmaryrd}
\usepackage{amsthm,amssymb}
\usepackage{etoolbox}
\usepackage{tikz}
\usepackage{amssymb}
\usepackage{tablefootnote}
\usetikzlibrary{matrix}
\usepackage{graphics,graphicx}
\usepackage[all]{xy}
\usepackage{url}
\usepackage{marginnote}
\definecolor{mygray}{gray}{0.85}
\usepackage[linecolor=black,backgroundcolor=mygray,colorinlistoftodos,prependcaption,textsize=small]{todonotes}
\usepackage{xargs}
\usepackage{xcolor}

\renewcommand{\leq}{\leqslant}

\newcommand\myrestriction{\mathord\restriction}
\def\mr#1{\myrestriction_{#1}}

\newcommand{\Mscr}{\mathcal M}

\newcommand{\Tscr}{\mathcal T}

\DeclareMathOperator{\pred}{pred}

\def\abar{\mbox{\boldmath $a$}}
\def\bbar{{\bf b}}
\def\cbar{{\bf c}}
\def\dbar{{\bf d}}
\def\ebar{{\bf e}}

\def\hbar{{\bf h}}

\def\mbar{{\bf m}}

\def\ubar{{\bf u}}
\def\vbar{{\bf v}}
\def\wbar{{\bf w}}
\def\xbar{{\bf x}}
\def\ybar{{\bf y}}
\def\zbar{{\bf z}}
\def\bb{{\bf \underaccent{\tilde}{b}}}
\def\VV{{\mathbb V}}
\def\forces{\Vdash}

\def\At{\rm{At}}

\newcommand{\NN}{\underaccent{\tilde}{N}}

\def\dom{{\rm dom}}

\def\tp{{\rm tp}}

\def\pcl{{\rm pcl}}

\def\<{\langle}
\def\>{\rangle}
\makeatletter
\def\myfnt{\ifx\protect\@typeset@protect\expandafter\footnote\else\expandafter\@gobble\fi}
\makeatother

\newtheorem{theorem}{Theorem}[section]

\newtheorem{corollary}[theorem]{Corollary}
\newtheorem{definition}[theorem]{Definition}

\newtheorem{lemma}[theorem]{Lemma}
\newtheorem{claim}[theorem]{Claim}

\newtheorem{proposition}[theorem]{Proposition}

\newtheorem{fact}[theorem]{Fact}

\newtheorem{remark}[theorem]{Remark}
\newtheorem{notation}[theorem]{Notation}

\newtheorem{definition/fact}[theorem]{Definition/Fact}

\def\bK{\mbox{\boldmath $K$}}

\newcommand{\pureindep}[1][]{%
  \mathrel{
    \mathop{
      \vcenter{
        \hbox{\oalign{\noalign{\kern-.3ex}\hfil$\vert$\hfil\cr
              \noalign{\kern-.7ex}
              $\smile$\cr\noalign{\kern-.3ex}}}
      }
    }\displaylimits_{#1}
  }
}

\newcommand{\indep}[2]{%
  \mathrel{
    \mathop{
      \vcenter{
        \hbox{%
\oalign{
\noalign{\kern-.3ex}\hfil$\vert$\hfil\cr
              \noalign{\kern-.7ex}
              $\smile$\cr\noalign{\kern-.3ex}
}
}
      }
}^{\!\!\!\!\!#2}_{\!\!\hspace{-0.1em}#1}
  }
}

\def\Mbar{{\overline{M}}}

\def\Q{{\mathbb Q}}
\def\P{{\mathbb P}}
\def\U{{\cal U}}
\def\V{{\mathbb V}}

\begin{document}

\title{When does $\aleph_1$-categoricity imply $\omega$-stability?
\thanks{All authors are
  grateful to an individual who prefers to remain anonymous' for providing typing services
 that were used during the work on the paper.  {\bf The second author is grateful to Danielle Ulrich for helpful discussions.}
  This is paper  1244 in the Shelah Archive.}}

\author{John T. Baldwin\thanks{Research partially supported by Simons
travel grant G3535}
\\University of Illinois at Chicago\\
\and Michael C. Laskowski\thanks{Partially supported by NSF grant
DMS-2154101.}\\University of Maryland\\
\and Saharon Shelah\thanks{The  third  author would like to thank the NSF and
BSF for partially supporting this research through  grant  NSF 2051825/ BSF
3013005232 with M. Malliaris. References like [Sh:950, Th0.2=Ly5] mean that
the internal label of Th0.2 is y5 in Sh:950. The reader should note that the
version in my website is usually more up-to-date than the one in arXiv. This
is publication number   1244 }}

  \maketitle

  \begin{abstract}  For an $\aleph_1$-categorical atomic class, we clarify the space of types over the unique model of size $\aleph_1$.
Using these results, we prove that if such a class has a model of size $\beth_1^+$ then it is $\omega$-stable.
\end{abstract}

  \section{Introduction}
  Our principal result is
  \begin{theorem}  \label{largemodeldisp}  If an atomic class $\At$ is $\aleph_1$-categorical
and has a model of size $(2^{\aleph_0})^+$, then $\At$ is $\omega$-stable.
\end{theorem}

This result springs from several related problems in the study of
$L_{\omega_1,\omega}$: the role of $\beth_{\omega_1}$,  the possible
necessity of the weak continuum hypothesis, the absoluteness of
$\aleph_1$-categoricity.

For first order logic, Morley \cite{Morley65} proved, enroute to his
categoricity theorem, that an $\aleph_1$-categorical first order theory is
$\omega$-stable (n\'e totally transcendental). The existence of a saturated
Ehrenfeucht-Mostowski model of cardinality $\aleph_1$ that is generated by a
well-ordered set of indiscernibles is crucial to the proof.   The
construction of such indiscernibles via the Erd\H{o}s-Rado theorem and
Ehrenfeucht-Mostowski models is tied closely to the existence of `large'
(i.e. of size $\beth_{\omega_1}$) models for the theory.


 The compactness of first order logic, yields  the full upward
L{\"o}wenheim-Skolem-Tarski (LST) theory for $L_{\omega,\omega}$: if $\psi$
has an infinite model it has arbitrarily large models. But for
$L_{\omega_1,\omega}$, the LST-theorem replaces `an infinite' by a model
 of size $\beth_{\omega_1}$. The proof proceeds by using iterations of the
 Erd\H{o}s-Rado theorem to find infinite sets of indiscernibles and to
 transfer size via Ehrenfeucht-Mostowski models.

 By an atomic class we mean the atomic models (Each finite sequence in each
model realizes a principal type over the empty set.) of a complete theory in
a countable first order language. Each sentence in $L_{\omega_1,\omega}$
defines such a class because Chang's theorem translates the sentence to a
first order theory omitting types and the language can be expanded to make
all realized types atomic \cite[Chapter 6]{Baldwincatmon}.

Shelah  calls an atomic class excellent if it satisfies an $n$-amalgamation
property for all $n$ and structures of arbitrary cardinality. He proved
\cite{Sh87a,sh87b} in ZFC:
 If an atomic class K is excellent and has an uncountable model then 1) it
has models of arbitrarily large cardinality; 2) if it is categorical in one
uncountable power it is categorical in all uncountable powers. He also
obtained a partial converse; under the very weak generalized continuum
hypothesis ($2^{\aleph_n} < 2^{\aleph_{n+1}}$ for $ n < \omega$): an atomic
class $\bK$ that has at least one uncountable model and is categorical in
$\aleph_n$ for each $ n < \omega$ is excellent. Thus the `Hanf number' for
existence  is reduced under VWGCH (and for categorical atomic classes) from
$\beth_{\omega_1}$ to $\aleph_{\omega}$.

This raises the question. Does an $\aleph_1$-categorical atomic class have
arbitrarily large models? Shelah \cite{Sh48} showed it has a model in
$\aleph_2$.

 For the authors, work on this problem began by searching for
sentences of $L_{\omega_1,\omega}$ for which $\aleph_1$-categoricity can be
altered by forcing.\footnote{For sentences of $L_{\omega_1,\omega}(Q)$, such
sentences exist, see \cite[\S 6]{Sh88}, expounded as \cite[\S
17]{Baldwincatmon}.  A   non-$\omega$-stable sentence with no models above
the continuum is given,
  where $\aleph_1$-categoricity fails under CH but holds under
Martin's Axiom.}
The third author
proposed an example, but the first author objected to the proof and the
second author proved in ZFC that the putative example was not
$\aleph_1$-categorical.

In \cite{BLSmanymod} we introduced the appropriate notion of an algebraic
type for atomic classes, {\em pseudo-algebraic} (Definition~\ref{pclextdef})
and proved there that for an atomic class with $< 2^{\aleph_1}$  models in
$\aleph_1$ the {pseudo-algebraic} types were dense.  In
\cite{LaskowskiShelahstrfail} the conclusion is strengthened to `pcl-small',
and here (assuming a model in $\beth_1^+$) to $\omega$-stability.

The search for weakened conditions for $\omega$-stability is partially
motivated by asking whether the absoluteness of $\aleph_1$-categoricity for
first order logic (given by the equivalence to $\omega$-stable and no
two-cardinal model) extends to atomic classes. \cite{Baldwinabscat} proves
that either arbitrarily large models ($\beth_{\omega_1}$) or
$\omega$-stability sufficed for such an absolute characterization. Our main
theorem reduces the $\beth_{\omega_1}$ to $\beth_1^+$.

In Section~\ref{ideaoutline} we  investigate several notions of {\em
constrained}, investigate their relation to $\omega$-stability and
$\aleph_1$-categoricity, and $\omega$-stablity. The notion of a constrained
type is just a renaming; a type $p\in S(M)$ is constrained just if it does
not split over a finite subset.  Such a type is definable in the standard use in
model theory -- the existence of a schema such that for all $\mbar \in M$,
$\phi(\xbar,\mbar ) \in p \leftrightarrow d_\phi(\xbar,\mbar)$.
In Section~\ref{sectwo1} we investigate many species of types over models and
see what happens under the assumption of $\aleph_1$-categoricity.
From this, we
prove the main theorem.  However, our results in Section~\ref{sectwo1} depend on
 a major hypothesis, the existence of an uncountable model in which every limit type is constrained.
  In
Section~\ref{secthree} we pay back our debt by proving
Theorem~\ref{constrainedexists}, we prove the existence of a model of size
$\aleph_1$ in which every limit type is constrained,
 using only the existence of an uncountable model.
Although the proof there uses forcing,  by appealing to the absoluteness
given by Keisler's model existence theorem for sentences of
$L_{\omega_1,\omega}(Q)$, the result is really a theorem of ZFC.


  \section{Constrained types, $\aleph_1$-categoricity and
  $\omega$-stability}\label{ideaoutline}
  \numberwithin{theorem}{subsection} \setcounter{theorem}{0}

Throughout this article, $T$ will denote a complete theory in a countable
language for which there is an uncountable atomic model.   $\At$ denotes the class of atomic models
of $T$.  In everything that
follows, we only consider atomic sets, i.e., sets for which every finite tuple is isolated
by a complete formula.    Throughout, $M,N$ denote atomic models and $A,B$ atomic sets.
We write $\abar,\bbar$ for finite atomic tuples, and $\xbar,\ybar,\zbar$ denote finite tuples of variables.

We repeatedly use the fact that the countable atomic model $M$ is unique up
to isomorphism.  Vaught  \cite{Vaught} showed the existence of an uncountable
atomic model is equivalent to the countable atomic model having a proper
elementary extension. The only types we consider are either over an atomic
model or are over a finite subset of a model.  In either case, we only
consider types realized in atomic sets.

 For general background see
\cite{Baldwincatmon} and more specifically \cite{BLSmanymod}.


\subsection{Constrained types and filtrations}

\begin{definition}\label{atomdef}\index{$S_{at}(A)$}\index{$S^*(A)$}
{\em Fix a countable complete theory $T$ with monster model $\Mscr$.
$At = At_T$
denotes the collection of atomic models of $T$.
\begin{enumerate}
\item  For $M \in At$, $S_{at}(M)$ is the collection of $p(\xbar) \in S(M)$ such
    that if $\abar\in \Mscr$ realizes $p$, $M\abar$ is an atomic set.
    \item $\At$ is {\em  $\omega$-stable} if for every/some countable $M\in\At$,
        $S_{at}(M)$ is countable.
        \end{enumerate}
        }
\end{definition}

The reader is cautioned that the definition of $\omega$-stability
is not equivalent to the classical notion (i.e., $S(M)$ countable) but within the context of atomic sets,
this revised notion of $\omega$-stability plays an analogous role.
The spaces $S_{at}(M)$ are typically not compact.  However, if $M$ is countable, then $S_{at}(M)$
is a $G_\delta$ subset of the full Stone space $S(M)$, and thus is a Polish space.  In particular,
if $\At$ is not $\omega$-stable, then $S_{at}(M)$ contains a perfect set.

\begin{definition}
\begin{enumerate}
\item  {\em  A type $p\in S_{at}(M)$ {\em splits over $F\subseteq M$} if
    there are tuples $\bbar,\bbar'\subseteq M$ and a formula
    $\phi(\xbar,\ybar)$ such that $\tp(\bbar/F)=\tp(\bbar'/F)$, but
    $\phi(\xbar,\bbar)\wedge\neg\phi(\xbar,\bbar')\in p$.

\item We call $p\in S_{at}(M)$ {\em constrained} if $p$ does not split over
    some finite $F\subseteq M$ and {\em unconstrained} if $p$ splits over
    every finite subset of $M$.

\item For any atomic model $M$, let $C_M:=\{p\in S_{at}(M): p \ \hbox{is
    constrained}\}$. We say $\At$ has {\em only constrained types} if
    $S_{at}(N)=C_N$ for every atomic model $N$.}

    \end{enumerate}
\end{definition}

We use the term constrained in place of `does not split over a finite subset' for its brevity, which is useful in subsequent definitions.

\begin{remark}

{\em The concepts in clauses (2) and (3) above
 give a method of proving that an atomic class is $\omega$-stable.
 $\At$ is $\omega$-stable holds if  a) $C_M$ is countable for some/every
countable atomic $M$ and  b) $\At$ has only constrained types. Immediately
$\omega$-stability implies a)
and the deduction
 of b) is standard \cite[Lemma 20.8]{Baldwincatmon}.
 Under the assumption of $\aleph_1$-categoricity, Theorem~\ref{CM} gives a)
and Theorem~\ref{omegastbchar} gives three equivalents of b).  However, the
short proof of Theorem~\ref{omegastbchar} makes crucial use of
Theorem~\ref{constrainedexists}, whose lengthy proof is relegated to
Section~\ref{secthree}.}


\end{remark}


The constrained types $p\in C_M$ are those that have a defining scheme over a
uniform finite set of parameters, i.e., if $p\in S_{at}(M)$ does not split
over $\abar$, then for every parameter-free $\phi(\xbar,\ybar)$, there is an
$\abar$-definable formula $d_p\xbar\phi(\xbar,\ybar)$ such that for any
$\bbar\in M^{|\ybar|}$, $\phi(\xbar,\bbar)\in p$ if and only if $M\models
d_p\xbar\phi(\xbar,\bbar)$.
We record three easy facts about extensions and restrictions of  types.


\begin{lemma}  \label{dropunconstrained}

\begin{enumerate}

\item  For any atomic models $M\preceq N$ and $A\subseteq M$ is finite, then for any  $q\in S_{at}(N)$ that does not split over $A$,
the restriction $q\mr{M}$ does not split over $A$; and any $p\in S_{at}(M)$ that does not split over $A$ has a unique non-splitting extension $q\in S{at}(N)$.


\item   If some atomic $N$ has an unconstrained $p\in S_{at}(N)$, then for every countable $A\subseteq N$, there is a countable $M\preceq N$
with $A\subseteq M$ for which the restriction $p\mr{M}$ is unconstrained.

\item  $\At$ has only constrained types if and only if $S_{at}(M)=C_M$ for every/some countable atomic model $M$.
\end{enumerate}
\end{lemma}

\begin{proof}  (1)  The first statement is immediate.  For the second, given
 $p(\xbar)\in S_{at}(M)$ non-splitting  over $A$, put $$q(\xbar):=\{\phi(\xbar,\bbar):\bbar\in N^{|\ybar|}, \phi(\xbar,\bbar')\in p
\ \hbox{for some
$\bbar'\in M$ with $\tp(\bbar'/A)=\tp(\bbar/A)$}\}$$

(2)  We construct $M\preceq N$ as the union of an increasing elementary $\omega$-chain $M_n\preceq N$ of countable, elementary substructures of $N$
with $A\subseteq M_0$ and, for each $n\in\omega$, $p\mr{M_{n+1}}$ splits over every finite $F\subseteq M_n$.  It follows that $M^*:=\bigcup\{M_n:n\in\omega\}$
is as required.

(3)  Left to right is immediate.  For the converse, assume there is some atomic $N$ with an unconstrained type $p\in S_{at}(N)$.  By (2) there is a countable $M\preceq N$
with $p\mr{M}$ unconstrained.
\end{proof}

Much of the paper concerns analyzing atomic models $N$ of size $\aleph_1$.   It is useful to consider any such $N$ as a direct limit of a family of countable, atomic submodels.

\begin{definition}  {\em  For $N$ of size $\aleph_1$, a {\em filtration of $N$} is a continuous, increasing sequence $(M_\alpha:\alpha\in\omega_1)$ of countable, elementary substructures.}
\end{definition}

When $N$ is atomic, then in any filtration $(M_\alpha:\alpha\in\omega_1)$ of $N$, each of the countable models are isomorphic.  As well, any two filtrations $(M_\alpha:\alpha\in\omega_1)$ and $(M_\alpha':\alpha\in\omega_1)$ agree on a club.  Thus, for any given countable $M\preceq N$, $\{\alpha\in\omega_1:M\preceq  M_\alpha$ and
$M_\alpha=M_\alpha'\}$ is club as well.

\subsection{$\aleph_1$-categoricity implies $C_M$ is countable
} \label{sectwo1}

Throughout this subsection, $\At$ is an atomic class that admits an
uncountable model and $M$ {\em  denotes a fixed copy of the countable atomic
model}. We aim to count the set  $C_M=\{p\in S_{at}(M):p$ is constrained$\}$.
Theorem~\ref{extra} yields the main result of the subsection:


\begin{theorem}  \label{CM}   If $\At$ is $\aleph_1$-categorical, then $C_M$ is countable for every/some countable atomic model $M$.
\end{theorem}

As $M$ is countable, the  natural action of $Aut(M)$ on the set $M$ induces
an action of $Aut(M)$ on $S_{at}(M)$.   When $M$ is atomic, a useful
characterization of $p \in C_M$ is: $C_M$
consists of those elements of $S_{at}(M)$ whose orbits are countable.
However, for the results in this section we  only require the easy half of
this statement.

\begin{lemma}  \label{alliso}   Suppose $p\in C_M$ and $M'$ is any countable, atomic model.
Then:
\begin{enumerate}
\item  $\{\pi(p):\pi:M\rightarrow M'$ an isomorphism$\}$ is a countable set
    of constrained types in  $S_{at}(M')$.
\item  There is a countable atomic $M^*\succ M'$ realizing $\pi(p)$ for
    every isomorphism $\pi:M\rightarrow M'$.
\end{enumerate}
\end{lemma}

\begin{proof}  (1) Choose a finite $A\subseteq M$ over which $p$ does not split.  As $M'$ is countable,
$A$ has only countably many images under isomorphisms $\pi:M\rightarrow M'$,
and it follows immediately from non-splitting that if
$\pi_1,\pi_2:M\rightarrow M'$ are isomorphisms satisfying $\pi_1(a)=\pi_2(a)$
for each $a\in A$, then $\pi_1(p)=\pi_2(p)$.

(2)  Using (1), let $\{q_i:i<\gamma\le\omega\}\subseteq S_{at}(M')$ be the
set of all images of $p$ under isomorphisms $\pi:M\rightarrow M'$. We
recursively construct an increasing sequence of countable models
$\{M_i:i<\gamma\}$ with $M_0=M'$ and, for each $i<\gamma$, $M_i$ contains a
realization of $q_j$ for every $j<i$.  Supposing $i<\gamma$ and $M_i$ has
been defined, let $q_i^*\in S_{at}(M_i)$ be the unique (\cite[Theorem
19.9]{Baldwincatmon}) non-splitting extension of $q_i\in S_{at}(M')$. Then
letting $d_i$ realize $q_i^*$, let $M_{i+1}\in At$ be an elementary extension
of $M_i$ containing $M_i\cup\{d_i\}$. Then $\bigcup_{i<\omega}M_i    $ works.
\end{proof}

\begin{definition}\label{RN}  {\em  Suppose  $(M_\beta:\beta<\omega_1)$
is a filtration of some $N\in\At$ of size $\aleph_1$.
For each $\beta<\omega_1$, let
$$R^\beta_N:=\{p\in C_M: \pi(p)\ \hbox{is realized in $N$ for every isomorphism $\pi:M\rightarrow M_\beta$}\}$$
and let $R_N:=\{p\in C_M:$ $p\in R^\beta_N$ for a stationary set of
$\beta\in\omega_1\}$. }
\end{definition}

As any two filtrations of $N$ agree on a club, it follows that $R_N$ is
independent of the choice of filtration of $N$. Similarly,  $R_N$ is an
isomorphism invariant, i.e., if $N\cong N'$ are each atomic models of size
$\aleph_1$, then $R_N=R_{N'}$.
We record two facts about $R_N$.

\begin{lemma}    \label{prev2}
\begin{enumerate}
\item  For any $N\in\At$ of size $\aleph_1$, $|R_N|\le \aleph_1$.
\item  For any $p\in C_M$ there is  some $N\in \At$ of size $\aleph_1$ such
    that $p\in R_N$.
\end{enumerate}

\end{lemma}

\begin{proof}   (1)  Choose any sequence  $\<p_i:i\in\omega_2\>$ from $R_N$ and we will show that $p_i=p_j$ for some distinct $i,j$.
Fix a filtration $(M_\alpha)$ of $N$. We shrink the sequence in two stages.
First, for each $i< \omega_2$, let $\alpha(i)\in\omega_1$ be least such that
$p_i\in R_N^{\alpha(i)}$. By pigeonhole and reindexing we may assume $\alpha(i)=\alpha^*$
for all $i$, i.e., each $p_i\in R^{\alpha^*}_N$. Now fix any isomorphism
$\pi:M\rightarrow M_{\alpha^*}$.  By definition of $R_N^{\alpha^*}$,
$\pi(p_i)$ is realized in $N$ for every $p_i$. But, as $|N|=\aleph_1$, there
is $c^*\in N$ realizing both $\pi(p_i)$ and $\pi(p_j)$ for some distinct
$i,j$.    Thus, $\pi(p_i)=\pi(p_j)$, hence $p_i=p_j$.

(2)  Fix $p\in C_M$.  Using Lemma~\ref{alliso}(2) at each level, construct a
continuous, increasing elementary sequence $M_\alpha$ of countable atomic
models such that, for every $\alpha<\omega_1$, $\pi(p)$ is realized in
$M_{\alpha+1}$ for every isomorphism $\pi:M\rightarrow M_\alpha$.  Put
$N:=\bigcup_{\alpha<\omega_1} M_\alpha$. Then $(M_\alpha)$ is a filtration of
$N$ and $p\in R^\alpha_N$ for every $\alpha<\omega_1$.  Thus, $p\in R_N$.
\end{proof}

%


%
We are now able to prove the theorem below, which clearly implies
Theorem~\ref{CM}.

\begin{theorem}  \label{extra}  If $C_M$ is uncountable, then $I(\At,\aleph_1)=2^{\aleph_1}$.
\end{theorem}

\begin{proof}    It is easily verified that $C_M$ is an $F_\sigma$ subset of the Polish space $S_{at}(M)$,
so on general grounds, $C_M$ is either countable or else it contains a
perfect set.

Our proof is non-uniform, depending on the  relative sizes of $2^{\aleph_0}$
and $2^{\aleph_1}$. First, under weak CH, i.e.,  $2^{\aleph_0}< 2^{\aleph_1}$
then combining arguments of  Keisler \cite{Keislerlq} and Shelah
\cite[Theorem 18.16]{Baldwincatmon} shows   if $I(\At,\aleph_1)\neq
2^{\aleph_1}$, then $\At$ is $\omega$-stable, so $S_{at}(M)$ is countable. As
$C_M\subseteq S_{at}(M)$, $C_M$ is countable as well.

On the other hand,
 assume  $2^{\aleph_0}=2^{\aleph_1}$, so in particular WCH fails.
 Under this assumption, we will prove that if $C_M$ is uncountable, then $I(\At,\aleph_1)=2^{\aleph_0}$,
which equals $2^{\aleph_1}$ under our cardinal hypotheses for this case.
Indeed, choose representatives $\{N_i:i\in\kappa\}$ for the isomorphism
classes of atomic models of size $\aleph_1$. If $C_M$ is uncountable, then
 the first sentence of the argument shows $|C_M|=2^{\aleph_0}$.  But by
Lemma~\ref{prev2}, $C_M\subseteq \bigcup \{R_{N_i}:i\in\kappa\}$ and
$|R_{N_i}|\le \aleph_1$ for each $i\in\kappa$. As we are assuming
$2^{\aleph_0}>\aleph_1$, we conclude $\kappa\ge 2^{\aleph_0}$, as required.
\end{proof}


\subsection{Limit types and $\aleph_1$-categoricity}

%
%

\begin{definition}\label{defconmod} {\em

A type  $p\in S_{at}(N)$ is a {\em limit type} if the restriction $p\mr{M}$
    is realized in $N$ for every countable $M\preceq N$.

}
\end{definition}

Trivially, for every $N$, every  type in $S_{at}(N)$ realized in $N$ is a limit type.
Since we allow $M = N$ in the definition of a limit type, if $M$ is countable, then
the only limit types in $S_{at}(M)$ are those realized in $M$.

Also, if $(M_\alpha:\alpha\in\omega_1)$ is a filtration of $N$, then a type $p\in S_{at}(N)$ is a limit type
if and only if $N$ realizes $p\mr{M_\alpha}$ for cofinally many $\alpha$.

 The long proof of the
 following crucial theorem is relegated  to
 Section~\ref{secthree}. Note that there are no additional
 assumptions on $\At$, other than the existence of an uncountable, atomic
 model.

\begin{theorem}  \label{constrainedexists}  If $\At$ admits an uncountable, atomic model,
 then there is some uncountable $N\in\At$ for which every limit type in $S_{at}(N)$ is constrained.
\end{theorem}

Here, we sharpen this result under the additional assumption of $\aleph_1$-categoricity.

\begin{corollary}  \label{constrained=limit}  If $\At$ is $\aleph_1$-categorical and $N\in\At$ has size $\aleph_1$,
then $C_N=\{\hbox{limit types in $S_{at}(N)$}\}$.
\end{corollary}

\begin{proof}  The hard direction of the equality is Theorem~\ref{constrainedexists}.   For the converse, first note that if every $p\in C_N$ is realized
in $N$, then each such $p$ is a limit type and we are done.  So, assume there is some $p\in C_N$ that is not realized in $N$.
By Lemma~\ref{dropunconstrained}(1),  $C_M$ contains a non-algebraic type for some countable $M\preceq N$.
By the uniqueness of countable atomic models, it follows that $C_M$ contains a non-algebraic type for every countable, atomic $M$.
As well, by Theorem~\ref{CM}, $C_M$ is countable for every such $M$.  Thus, every countable atomic $M$ has a proper, countable elementary extension
$M'\succ M$ containing a realization of every $p\in C_M$.  Iterating this $\omega_1$ times, we obtain a model $N=\bigcup\{M_\alpha:\alpha\in\omega_1\}$
of size $\aleph_1$ such that $N$ realizes  every $p\in C_{M_\alpha}$ for every $\alpha\in \omega_1$.  It follows that for this $N$, every $q\in C_N$ is a limit type.
Indeed, suppose $q$ does not split over $A$.  Given any countable $M\preceq N$, choose
 $\alpha$ such that $M\cup A\subseteq M_\alpha$.  By  Lemma~\ref{dropunconstrained}(1)
 $q\mr{M_\alpha}\in C_{M_{\alpha}}$,  hence both it and therefore
$q\mr{M}$, are realized in $N$.  Finally, since $\At$ is $\aleph_1$-categorical, every $N$ of size $\aleph_1$ has this property.
\end{proof}


\subsection{Characterizing $\omega$-stability}\label{charwstb}
  In this Subsection, we first derive Lemma~\ref{omegastable}
that gives three consequences of $\omega$-stability in terms of the behavior
of constrained types.  Then, taking Theorem~\ref{constrainedexists} as a
black box (proved in Section~\ref{secthree}), Lemma~\ref{omegastbchar} shows
that each of these conditions is equivalent to $\omega$-stability under the
assumption of $\aleph_1$-catgoricity. Finally, Theorem~\ref{largemodel}
asserts that the existence of a model in $\beth_1^+$ and
$\aleph_1$-categoricity implies condition 1) of Lemma~\ref{omegastbchar} and
thus $\omega$-stability.


\begin{definition}  {\em
\begin{itemize}
\item  A {\em proper constrained pair} is a  pair $N \neq N'$ of atomic
    models such that $\tp(\cbar/N)$ is constrained for every tuple
    $\cbar\in N'$.
\item A  {\em proper relatively $\aleph_1$-saturated pair} is a proper pair
    $N\neq N'$ such that, for every countable $M\preceq N$, every type
    $p\in S(M)$
realized in $N'$ is realized in $N$.
\end{itemize}
}
\end{definition}

Note that in (2), both models must be uncountable, whereas (1) makes sense for countable models as well.
Of course, in (2) it would be equivalent to say that `every type over every
countable set $A\subseteq N$ that is realized in $N'$ is realized in
$N$,' but we choose the definition above to conform with our convention
about only looking at
 types over models.

 \begin{lemma}  \label{transitive}  Let $\At$ be any atomic class.
 \begin{enumerate}
 \item  If both $(M,M')$ and $(M',M'')$ are constrained pairs, then $(M,M'')$ is a constrained pair as well.
 \item  If $(M,M')$ is a constrained pair of countable atomic models, then there is an uncountable $N$ with a filtration
 $(M_\alpha:\alpha\in\omega_1)$ such that $(M_\alpha,N)$ is a constrained pair for every $\alpha\in\omega_1$.
 \end{enumerate}
 \end{lemma}

 \begin{proof}  (1)  Choose any $\cbar\in M''$.  As $(M',M'')$ is a constrained pair, choose $\bbar\in M'$ such that
 $\tp(\cbar/M')$ does not split over $\bbar$.  As $(M,M')$ is a constrained pair, choose $\abar\in M$ such that $\tp(\bbar/M)$ does
 not split over $\abar$.  We claim that $\tp(\cbar\bbar/M)$ does not split over $\abar$, which clearly suffices.
 To see this, choose any $\mbar_1,\mbar_2$ from $M$ such that $\tp(\mbar_1\abar)=\tp(\mbar_2\abar)$.
 By non-splitting, this implies $\tp(\mbar_1\abar\bbar)=\tp(\mbar_2\abar\bbar)$.  Now both $\mbar_1\abar$ and $\mbar_2\abar$ are from $M'$,
 hence $\tp(\mbar_1\abar\bbar\cbar)=\tp(\mbar_2\abar\bbar\cbar)$ as $\tp(\cbar/M')$ does not split over $\bbar$.

 (2)  As $M$ is a countable atomic model that is the lower part of a constrained pair, so is any other countable, atomic model.
 Thus, we can form a continuous, increasing chain $(M_\alpha:\alpha\in\omega_1)$ of countable atomic models with $(M_\alpha,M_{\alpha+1})$
 a constrained pair for each $\alpha$.  This chain is a filtration of the atomic $N:=\bigcup\{M_\alpha:\alpha\in\omega_1\}$.  That each $(M_\alpha,N)$
 is a constrained pair follows from (1).
 \end{proof}

We record the following consequences of $\omega$-stability in atomic classes.  It is
noteworthy that $\aleph_1$-categoricity plays no role in
Lemma~\ref{omegastable}, and without additional assumptions, none of these imply $\omega$-stability.
However, following this, with Theorem~\ref{omegastbchar} we see that when coupled with $\aleph_1$-categoricity, each of these conditions implies $\omega$-stability.

\begin{lemma} \label{omegastable} Suppose $\At$ is an $\omega$-stable atomic class that admits an uncountable atomic model.
Then:
\begin{enumerate}
\item  $\At$ has only constrained types;
\item  $\At$ has a proper constrained pair; and
\item  $\At$ has a proper, relatively $\aleph_1$-saturated pair.
\end{enumerate}
\end{lemma}

\begin{proof} (1)  For an $\omega$-stable atomic class, one can define
(\cite[Definition 19.1]{Baldwincatmon})
 a splitting rank on types $p\in S_{at}(N)$ for any model $N$ such that (\cite[Theorem 19.8]{Baldwincatmon}):
for any atomic model $N$ and any $p\in S_{at}(N)$, then choosing
$\phi(x,\abar)\in p$ to be a complete formula of smallest rank,  $p$ does not
split over $\abar$.  That is, $p$ is constrained.

(2)  Choose any countable, atomic model $M$.  Since $\At$ admits an
uncountable atomic model, there is a countable, proper, atomic elementary
extension $M'\succ M$. By (1), $\tp(c/M)$ is constrained for every $c\in M'$,
hence $(M,M')$ is a proper constrained pair.

(3)  We first argue that there is an {\em atomically saturated} model $N$ of
size $\aleph_1$.  That is, for every countable $M\preceq N$, $N$ realizes
every $p\in S_{at}(M)$. The existence of an uncountable, atomically saturated
$N$ is easy; build a  union of a continuous elementary chain $
(M_\alpha:\alpha\in\omega_1)$ of countable atomic models with the
property that  for each $\alpha<\omega_1$, $M_{\alpha+1}$ realizes
 every $p\in S_{at}(M_\alpha)$.  The existence of such an $M_{\alpha+1}$ is immediate since $S_{at}(M_\alpha)$
is countable and  every $p\in S_{at}(M_\alpha)$ can be realized in some
countable, atomic elementary extension.

Now, given an atomically saturated model $N$ of size $\aleph_1$, recall
that if $\At$ is $\omega$-stable, then every model of size $\aleph_1$ has a
 proper atomic extension $N'$, see e.g., the proof of 19.26 of \cite{Baldwincatmon}.  But then $(N,N')$ is a proper, relatively $\aleph_1$-saturated pair.
\end{proof}

%
%
%

Now, given Theorem~\ref{CM} and  Corollary~\ref{constrained=limit} (the latter depending on the promised
Theorem~\ref{constrainedexists}) we give short proofs of our main results.
The key idea is that from Theorem~\ref{constrainedexists}, we know there is a
constrained model in $\aleph_1$; while from 1) and not 2) we find an
unconstrained one.

\begin{theorem}\label{omegastbchar}  The following are equivalent for an
$\aleph_1$-categorical atomic class $\At$.
\begin{enumerate}
\item  $\At$ has a proper, relatively $\aleph_1$-saturated pair;
\item  $\At$ has a proper constrained pair;
\item  $\At$ has only constrained types; and
\item  $\At$ is $\omega$-stable.
\end{enumerate}
\end{theorem}

\begin{proof}  We will show $(1)\Rightarrow(2)\Rightarrow(3)\Rightarrow(4)$, which in light of Lemma~\ref{omegastable} suffices.

$(1)\Rightarrow(2):$  Suppose $(M^*,M^{**})$ is a proper, relatively
$\aleph_1$-saturated pair of atomic models, and by way of contradiction
suppose that $(M^*,M^{**})$ is not a proper constrained pair.  Choose $c\in M^{**}$
such that
 $p:=\tp(c/M^*)$ is unconstrained.
Then, by iterating Lemma~\ref{dropunconstrained}(2), we construct a continuous,
 elementary chain $(M_{\alpha}:\alpha\in\omega_1)$ of countable, elementary substructures
of $M^*$ such that, for every $\alpha\in \omega_1$, $p\mr{M_{\alpha}}$ is
unconstrained, but is realized in $M_{\alpha+1}$. To accomplish this, by
Lemma~\ref{dropunconstrained}(2), choose a countable $M_0\preceq M^*$ such
that $p\mr{M_0}$ is unconstrained.  At countable limits, take unions.
Finally, given a countable $M_\alpha\preceq M^*$, by relative
$\aleph_1$-saturation choose $c_\alpha\in M^*$ realizing $p\mr{M_\alpha}$ and
then apply Lemma~\ref{dropunconstrained}(2) to the set
$M_\alpha\cup\{c_\alpha\}$ to get $M_{\alpha+1}\preceq M^*$ with
$p\mr{M_{\alpha+1}}$ unconstrained.  Let
$N:=\bigcup\{M_\alpha:\alpha\in\omega_1\}$. Then $N$ has size $\aleph_1$ and
the type $p\mr{N}$ is an unconstrained limit type, contradicting
$\aleph_1$-categoricity by Corollary~\ref{constrained=limit}.

$(2)\Rightarrow(3):$  Assume that  $(N^*,N^{**})$ is a
proper constrained pair (of any cardinality).
 By an easy L\"owenheim-Skolem
argument (in the pair language) there is a proper constrained pair $(M,M')$   of countable atomic models.
By Lemma~\ref{transitive}(2), there is an atomic model $N$ of size $\aleph_1$ with a filtration $(M_\alpha:\alpha\in\omega_1)$
such that $(M_\alpha,N)$ is a constrained pair for every $\alpha\in\omega_1$.

Now, by way of contradiction, assume (3) fails.
By Lemma~\ref{dropunconstrained}(3), $S_{at}(M)$ contains an unconstrained type for every countable atomic model $M$.
Thus, for any such $M$, there is a countable atomic $M'\succ M$ containing a realization of an unconstrained type.
By iterating this $\omega_1$ times, we construct a continuous, elementary chain $(M_\alpha':\alpha\in\omega_1)$ for
which $M_{\alpha+1}'$ contains a realization of an unconstrained type in $S_{at}(M_{\alpha}')$.
Let $N':=\bigcup\{M_\alpha':\alpha\in\omega_1\}$.  Note that $(M_\alpha',N')$ is never a constrained pair.
But this contradicts $\aleph_1$-categoricity:  If $f:N\rightarrow N'$ were an isomorphism, then there would be (club many)
$\alpha\in \omega_1$ such that $f\mr{M_\alpha}$ maps $M_\alpha$ onto $M_\alpha'$, hence maps the pair $(M_\alpha,N)$ onto $(M_\alpha',N')$.
As the former is a constrained pair, while the latter is not, we obtain a contradiction.

 $(3)\Rightarrow(4):$  Assume $\At$ has only constrained types and let $M$ be any countable, atomic model.  This means that $S_{at}(M)=C_M$.
 However, as $\At$ is $\aleph_1$-categorical, $C_M$ is countable by Theorem~\ref{CM}.  Thus, $S_{at}(M)$ is countable, which is the definition of $\At$ being $\omega$-stable.
\end{proof}

With this result in hand, it is easy to deduce the main theorem.
This is the {\em only} use of the existence of a model in $\beth_1^+$.   We
imitate the classical proof that every $\kappa\ge |L|$, every $L$-theory with
an infinite model has a $\kappa^+$-saturated model of size $2^\kappa$, to
prove clause 1) of Lemma~\ref{omegastbchar} and thus deduce
$\omega$-stability.

\begin{theorem}  \label{largemodel}  If an atomic class $\At$ is $\aleph_1$-categorical
and has a model of size $(2^{\aleph_0})^+$, then $\At$ is $\omega$-stable.
\end{theorem}

\begin{proof}  Let $M^{**}$ be an atomic model of size $(2^{\aleph_0})^+$.
We construct a  relatively $\aleph_1$-saturated elementary substructure
$M^*\preceq M^{**}$ of size $2^{\aleph_0}$ as the union of a continuous chain
$(N_\alpha:\alpha\in\omega_1)$ of elementary substructures of
$M^{**}$, each of size $2^{\aleph_0}$, where, for each $\alpha<\omega_1$ and
each of the $2^{\aleph_0}$ countable $M\preceq N_{\alpha}$, $N_{\alpha+1}$
realizes each of the at most $2^{\aleph_0}$ $p\in S(M)$ that is realized in
$M^{**}$. $\omega$-stability is immediate from $(1)\Rightarrow(4)$ in
Lemma~\ref{omegastbchar}.
\end{proof}

\section{Paying our debt}\label{secthree}

The whole of this section is aimed at proving
Theorem~\ref{constrainedexists}:  If an $L_{\omega_1,\omega}$ sentence has an
uncountable model it has one in which every limit type is constrained.
 The proof relies heavily on Keisler's completeness theorem  that implies
 `model existence' of sentences
of $L_{\omega_1,\omega}(Q)$ is absolute between forcing extensions. In the
first subsection, we explicitly give an $L_{\omega_1,\omega}(Q)$ sentence
$\Psi^*$ in a countable language extending the language of $\At$ such that in
any set-theoretic universe, $\Psi^*$ has a model of size $\aleph_1$ if and
only if there is an atomic model of
size $\aleph_1$ with every limit type constrained.

The second subsection describes   family of striated formulas
(\cite{BLSmanymod}).
 Such formulas are used to describe a c.c.c.\ forcing notion $(\P,\le)$ in the third subsection.  There, we prove that $(\P,\le)$ forces the existence of an atomic of size $\aleph_1$
 with every limit type constrained.   Thus, we conclude that $\Psi^*$ has a model of size $\aleph_1$ in a c.c.c.\ forcing extension, so by the absoluteness described above,
$\V$ has a model of $\Psi^*$ of size $\aleph_1$, yielding our requested model.

\subsection{Finding a requisite sentence $\Psi^*$ of  $L_{\omega_1,\omega}(Q)$}
\numberwithin{theorem}{subsection} \setcounter{theorem}{0}

This subsection is devoted to proving the following Proposition.

\begin{proposition}  \label{main} Let $T$ be  a first order $L$-theory for a
 countable language and $N$ an
atomic model  of $T$ with $|N| =\aleph_1$. There is a sentence $\Psi^*\in
(L^*)_{\omega_1,\omega}(Q)$ in
an expanded (but still countable) language $L^*\supseteq L$ such that:\\
$(\star)$ There is a model $N^*\models \Psi^*$ with $|N|=\aleph_1$ with a
 definable linear order $(I,\le_I)$ of cofinality $\aleph_1$ if and only if
 the $L$-reduct $N\in\At$ has every limit type constrained.
\end{proposition}

As we will be interested in arbitrary models  of a sentence and because ``is
a well ordering'' is not
 expressible in $L_{\omega_,\omega}(Q)$, we need to generalize the notion of a filtration.

\begin{definition}  {\em  A linear order $(I,\le)$ is {\em $\omega_1$-like} if it
 has cardinality $\aleph_1$, but, letting  $\pred(i)$ denote $\{j\in I:
 j<i\}$, for
 every $i\in I$,
 $|\pred(i)|\leq \aleph_0$.

If $N$ is any set of size $\aleph_1$ and $(I,\le)$ is $\omega_1$-like, then an
{\em $(I,\le)$-scale} is a surjective function $f:N\rightarrow I$ such that $f^{-1}(i)$ is countable for every $i\in I$.

If $f:N\rightarrow I$ is a scale, put $A_i:=f^{-1}(\pred(i))$ for every $i\in
I$, and note that each $A_i$ is countable.

}

\end{definition}

 We consider the
sets $(A_i:i\in I)$ to be a surrogate for a filtration of $N$;  $A_i$
replaces $M_\alpha$.  We now define a tree order on types over certain
countable subsets of a model of $T$ with cardinality $\aleph_1$.


\begin{definition}\label{scaledef}{\em Fix $T,N$ as in Proposition~\ref{main}.
 Suppose $(I,\le)$ is an $\omega_1$-like linear order and $f:N\rightarrow I$ is a scale.
\begin{enumerate} \item Define an equivalence relation $E_f$ on $(N\times I)$
 as  $(a,i)E_f(b,j)$ if and only if  $i=j$  and $\tp(a/A_i)=\tp(b/A_i)$.
 Thus each equivalence class corresponds to a type.

\item Define a strict partial order $\prec_f$ on $(N\times I)/E_f$ as:
    $[(a,i)]\prec_f [(b,j)]$ if and only if $i<_I j$;
 $\tp(a/A_i)=\tp(b/A_i)$; and $\tp(b/A_j)$ splits over every finite
    $F\subseteq A_i$.

    \item A {\em $\prec_f$-chain} is a sequence of types linearly ordered
        by $\prec_f$ (hence splitting).
\end{enumerate}}
\end{definition}

It is evident that $((N\times I)/E,\prec_f)$ is {\em tree-like} in that
 the $\prec_f$-predecessors
of every $E_f$-class are linearly ordered by $\prec_f$.  Moreover, since
$(I,\le)$ is $\omega_1$-like, every $E$-class has only countably
 many $\prec_f$-predecessors.

\begin{lemma} \label{connect} Let $N$ be any atomic model of size $\aleph_1$,
 $(I,\le)$ be $\omega_1$-like,
 $f:N\rightarrow I$ be any scale and $I$, $E_f$, $A_i$, and $\prec_f$ be as
in Definition~\ref{scaledef}. The following are equivalent:
 \begin{enumerate} \item there exists an
$f$ such that $\Tscr_f=((N\times I)/E_f,\prec_f)$ has an uncountable
$\prec_f$-chain;
\item Some limit type in $S_{at}(N)$ is unconstrained;
 \item for every $f$,  $\Tscr_f=((N\times I)/E_f,\prec_f)$ has an
     uncountable $\prec_f$-chain.
\end{enumerate}
\end{lemma}

\begin{proof}  $(3)\Rightarrow(1)$ is immediate.  For $(1)\Rightarrow(2)$,
 suppose for some $f$, $C\subseteq \Tscr_f$ is an uncountable
 $\prec_f$-chain.  As $[(a,i)]\prec_f[(b,j)]$ implies $i<j$ and
 since $(I,\le)$ is $\omega_1$-like,
$\pi_2(C):=\{i\in I:\exists a\in N [(a,i)]\in C\}$ is cofinal in $I$, hence
$\bigcup\{A_i:i\in\pi_2(C)\}=N$. Also, as $[(a,i)]\prec_f[(b,j)]$ implies
$\tp(a/A_i)=\tp(b/A_i)$, there is a unique $p\in S_{at}(N)$ defined as
$p:=\bigcup\{\tp(a/A_i):(a,i)\in C\}$. Furthermore, as $\tp(b/A_j)$ splits
over every finite $F\subseteq A_i$, it follows that $p$ is unconstrained.
Recalling Definition~\ref{defconmod}(2), it remains to show that $p$ is a
limit type. Choose a filtration $\Mbar=(M_\alpha)$ of $N$ and  argue that
$p\mr {M_\alpha}$ is realized in $N$ for every $\alpha\in \omega_1$.  Given
$\alpha\in\omega_1$, choose $i\in\pi_2(C)$ such that $M_\alpha\subseteq A_i$.
Then each $a\in N$ for which $(a,i)\in C$ realizes $p\mr{A_i}$ and hence
realizes $p\mr{M_\alpha}$. So $p$ is a limit type.

$(2)\Rightarrow(3)$.  Suppose $N$ has an unconstrained limit type $p\in S_{at}(N)$ and fix a scale $f$.  Also choose a filtration
$(M_\alpha:\alpha\in \omega_1)$ of $N$.
To construct an uncountable chain $\Tscr_f$ we repeatedly use the following
claim.

\medskip
\begin{claim}\label{cl}
 For every countable $B\subseteq N$ there is $i\in I$ such that
\begin{itemize}
\item  $B\subseteq A_i$;
\item $p\mr{A_i}$ is realized; and
\item  $p\mr{A_i}$ splits over every finite $F\subseteq B$.
\end{itemize}
\end{claim}

\begin{proof}  Given a countable $B\subseteq N$, since $p\in S_{at}(N)$ splits over every finite $F\subseteq N$, there is a countable $B^*\supseteq B$
such that $p\mr{B^*}$ splits over every finite $F\subseteq B$.  Now choose
$i\in I$ such that $B^*\subseteq A_i$ and then choose $\alpha\in \omega_1$
such that $A_i\subseteq M_\alpha$.  Since $p$ is a limit type, choose $c\in
N$ realizing $p\mr{M_\alpha}$ and hence $p\mr{A_i}$.
\end{proof}

Iterating Claim~\ref{cl} $\omega_1$ times yields a strictly increasing
sequence $(i_\alpha:\alpha\in \omega_1)$ from $(I,\le)$ and
$(c_\alpha:\alpha\in\omega_1)$ from $N$, where at each stage $\alpha$, we
take $B=\bigcup\{A_{i_\beta}:\beta<\alpha\}$. It follows directly from the
definition of $\prec_f$ that $(c_\beta,i_\beta)\prec_f(c_\alpha,i_\alpha)$
whenever $\beta<\alpha$, so $((N\times I)/E,\prec_f)$ has an uncountable
chain.
\end{proof}

With  Lemma~\ref{connect} in hand, we now define the sentence $\Psi^*$
described in Proposition~\ref{main}.

\begin{definition}  {\em  Let $L^*:=L\cup\{I,\le_I,f,E,\prec_f\}\cup\{\Q,\le_{\Q},H\}$ and let $\Psi^*$ be a
set of  $L_{\omega_1,\omega}(Q)$-axioms ensuring that for any $N^*\models\Psi^*$:
\begin{enumerate}
\item  The $L$-reduct of $N$ is in $\At$;
\item  $N^*$ is uncountable;
\item  $I\subseteq N$ and $(I,\le_I)$ is an $\omega_1$-like linear order;
\item  $f:N\rightarrow I$ is a scale; (recall: $A_i:=f^{-1}(\pred(i))$);
\item  $E\subseteq N\times I$ satisfies $(a,i)E(b,j)$ iff $i=j$ and $\tp(a/A_i)=\tp(b/A_i)$;
\item  For all $[(a,i)],[(b,j)]\in (N\times I)/E$, $[(a_i)]\prec_f[(b,j)]$
    iff $i<j$; $\tp(a/A_i)=\tp(b/A_i)$; and $\tp(b/A_j)$ splits over every
    finite $F\subseteq A_i$;
\item  $\Q\subseteq N$ and $(\Q,\le_{\Q})$ is a countable model of DLO;
\item  $H:N\times I\rightarrow \Q$ satisfies: For all $(a,i),(b,j)$,
\begin{enumerate}
\item  If  $(a,i)E(b,j)$ then $H(a,i)= H(b,j)$; and
\item  If $[(a,i)]\prec_f[(b,j)]$, then $H(a,i)<_{\Q} H(b,j)$.
\end{enumerate}
\end{enumerate}
}
\end{definition}

Relying on the general Lemma~\ref{suslin}, we have:

\noindent{\bf Proof of Proposition~\ref{main}}

First, assume $N^*\models \Psi^*$ and let $N$ be the $L$-reduct of $N^*$. As
the ordering on $(\Q,\le)$ forbids a strictly increasing $\omega_1$ sequence,
the existence of the function $H$ forbids $T=((N\times I)/E,\prec_f)$ having
an uncountable $\prec_f$-chain. Thus, by Lemma~\ref{connect},
every limit type in  $S_{at}(N)$ is constrained.


The converse is more involved.  Assume we are given  $N\in\At$
of size $\aleph_1$ with every limit type in $S_{at}(N)$ constrained.
  We can choose arbitrary subsets $I,\Q\subseteq N$ of
cardinality $\aleph_1,\aleph_0$, respectively and choose orderings $\le_I$
and $\le_\Q$ as required by $\Psi^*$.  Fix an arbitrary scale $f:N\rightarrow
N$ and intepret $E$ and $\prec_f$ as required. It only remains to find a
function $H:N\times I\rightarrow\Q$ as requested by $\Psi^*$.  For this, we
turn to forcing, and show in Lemma~\ref{suslin} that since there is no
uncountable $\prec_f$-chain in $N$, such an interpretation of $H$ exists in
$\VV[G]$, a generic extension of $\VV$ by a c.c.c.\ forcing. Note that this
completes the proof of Proposition~\ref{main}. Indeed, we have shown that
there is a model of $\Psi^*$ in some forcing extension. Thus,
$$\VV[G]\models \hbox{`There is a model of $\Psi^*$'}$$
Hence, by Keisler's model existence theorem for sentences of
$L_{\omega_1,\omega}(Q)$, there is a model $M^{**}\models \Psi^*$ in $\VV$.
Although the $L$-reduct of this model may be very different than our initial
$N$, this suffices. $\blacksquare$

\medskip

We prove the remaining Lemma in a very general form replacing $\prec_f$ by a
general tree as in Definition~\ref{gentree}. If one prefers, one could take
the function into $(\Q,\le)$ and be increasing whenever $a\prec b$ holds, but
the formulation below seems more basic.  The form of the sentence $\Psi^*$
involving maps
 into $(\Q,\le)$ is  more like classical definitions of special Aronszajn trees \cite{Sh73}.

\begin{definition} \label{gentree} {\em  By a {\em tree-like strict partial order} $(X,\prec)$ we mean a strict partial order with $|X|=\aleph_1$
such that for every $a\in X$, the induced $(pred(a),\prec)$ is a countable
linear order. }
\end{definition}

\begin{lemma}\label{suslin}  Suppose $(X,\prec)$ is any tree-like strict partial order with no uncountable chain.
Then there is a c.c.c.\ forcing $(\P,\le)$ such that in any generic $\V[G]$ there is a function $H:X\rightarrow\omega$
such that if $a\prec b$, then $H(a)\neq H(b)$.
\end{lemma}

\begin{proof}  The partial order $(\P,\le)$ is simply the set of all finite
approximations of such an $H$.  That is, $\P$ is the set of all functions
$h:X_0\rightarrow \omega$ with $X_0\subseteq X$ finite such that for all
$a,b\in X_0$, if $a\prec b$, then $h(a)\neq h(b)$, ordered by inclusion,
i.e., ($\sharp$)  $h\le h'$ if and only if $h\subseteq h'$. It is easily
checked that this forcing will produce (in $\V[G]$) a total function
$H:X\rightarrow\omega$ as desired.  The non-trivial part is showing that
$(\P,\le)$ has the c.c.c. For this, choose any uncountable set
$Y=\{h_\alpha:\alpha\in\omega_1\}\subseteq \P$ and assume, by way of
contradiction, that $h_\alpha\cup h_\beta\not\in\P$ for distinct
$\alpha,\beta\in\omega_1$. By passing to a subset of $Y$, we may assume
$|\dom(h_\alpha)|=n$ for some fixed $n\in\omega$ and we argue by
contradiction.  If $n=1$, i.e., $\dom(h_\alpha)=\{a^\alpha\}$, then by
passing to a further subset,
 there is a single $m^*\in\omega$ such that $h_\alpha(a^\alpha)=m^*$ for every $\alpha$.
The only way we could have $h_\alpha\cup h_\beta\not\in\P$ would be if
 $a^\alpha,a^\beta$ were distinct, but $\prec$-comparable.  But then
 $C=\{a^\alpha:\alpha\in\omega_1\}$ would
be an infinite chain in $(X,\prec)$, contradicting our assumption.

So, assume $n>1$ and we have proved   (c.c.c.) for all $n'<n$.  To ease
notation, enumerate the universe $X$ with order type $\omega_1$.  For each
$\alpha$, write $\dom(h_\alpha)=(a^\alpha_1,\dots,a^\alpha_n)$ in increasing
order, subject to this enumeration. By the $\Delta$-system lemma, there is an
uncountable subset and a root $r$ such that
$\dom(h_\alpha)\cap\dom(h_\beta)=r$ for all distinct pairs $\alpha,\beta$. If
$r\neq\emptyset$, we can apply our inductive hypothesis to the family of sets
$\{\dom(h_\alpha)\setminus r:\alpha\in\omega_1\}$, so we may assume
$r=\emptyset$, i.e., the domains $\{\dom(h_\alpha):\alpha\in\omega_1\}$ are
pairwise disjoint.  Again, passing to a subsequence, we may assume that with
respect to the global enumeration of $X$ $a^\alpha_n<a^\beta_1$ for all
$\alpha<\beta$.  Additionally, we may assume there are integers
$(m_1,\dots,m_n)$ such that $h_\alpha(a^\alpha_i)=m_i$ for all
$\alpha\in\omega_1$ and $i\in\{1,\dots,n\}$.

Now fix $\alpha<\beta$.  In order for $h_\alpha\cup h_\beta$ to not be in
$\P$, there must be some $p(\alpha,\beta),q(\alpha,\beta)\in\{1,\dots,n\}$
such that $a^\alpha_{p(\alpha,\beta)}$ and $a^\beta_{q(\alpha,\beta)}$ are
$\prec$-comparable. As a bookkeeping device, fix a uniform\footnote{That is,
every $Y\in\U$ has cardinality $\aleph_1$.  Equivalently,
 $\U$ contains all of the co-countable subsets of $\omega_1$.} ultrafilter $\U$ on $\omega_1$.

Thus, for any $\alpha\in\omega_1$, there is some $S_\alpha\in \U$, some
$p(\alpha), q(\alpha)\in\{1,\dots,n\}$ such that, by ($\sharp$), for every
$\beta\in S_\alpha$, $a^\alpha_{p(\alpha)}$ and $a^\beta_{q(\alpha)}$ are
$\prec$-comparable. However, since $pred(a^\alpha_{p(\alpha)})$ was assumed
to be countable, there is $S^*_\alpha\subseteq S_\alpha$, $S_\alpha^*\in\U$
such that
 $a^\alpha_{p(\alpha)}\prec a^\beta_{q(\alpha)}$ for all $\beta\in S^*_\alpha$.

Similarly, there is some $S\in\U$ and some $p^*,q^*\in\{1,\dots,n\}$ such
that for all $\alpha\in S$ and for all $\beta\in S^*_\alpha$ we have
$a^\alpha_{p^*}\prec a^\beta_{q^*}$. We obtain our contradiction by showing
that $$C=\{a^\alpha_{p^*}:\alpha\in S\}$$ is an uncountable chain in
$(X,\prec)$. Since $\U$ is uniform, $C$ is uncountable.  To get
comparability, choose any $\alpha,\gamma\in S$.  As $S^*_\alpha,S^*_\gamma\in
\U$, there is $\beta\in S^*_\alpha\cap S^*_\gamma$.  It follows that
$a^\alpha_{p^*}\prec a^\beta_{q^*}$ and $a^\gamma_{p^*}\prec a^\beta_{q^*}$.
As $(X,\prec)$ was assumed to be tree-like, it follows that $a^\alpha_{p^*}$
and $a^\gamma_{p^*}$ are $\prec$-comparable.
\end{proof}

\subsection{Extendible and striated formulas}\label{exstri}

Throughout this section, we work with the atomic models of a complete, first
 order theory $T$ in a countable language that has an uncountable atomic model.
 We expound model theoretic properties needed in the forcing construction of
 Section~\ref{theforcing}.

\begin{remark} {\em  In this section we  work  with complete
 formulas $\theta(\wbar)$, usually with a prescribed partition of the free variables.
Regardless of the partition, for any subsequence $\vbar\subseteq \wbar$,
 we use the notation $\theta\mr{\vbar}$ to denote the complete formula in the variables
$\vbar$ that is equivalent to $\exists \ubar \theta(\vbar,\ubar)$ where $\ubar=(\wbar\setminus\vbar)$.
}
\end{remark}

\begin{definition}\label{pclextdef}
{\em
\begin{enumerate}
\item A complete formula $\phi(x,\abar)$ is {\em
    pseudo-algebraic\footnote{The careful distinctions of
    pseudoalgebraicity `in a model' of \cite{BLSmanymod} are avoided
    because we have assumed there is an uncountable atomic model.}} if for
    some/any countable $M$ with $\abar \in M$ and any $N \succneqq M $,
    $\phi(N,\abar) = \phi(M,\abar)$.

\item $b \in \pcl(\abar,M)$, written $b \in \pcl(\abar)$ if and only if
    every $N\preceq M$ with $\abar \in N$, $b \in N$.

\item  A complete formula $\theta(\zbar;\xbar)$ is {\em extendible} if
    there is a pair $M\preceq N$ of countable, atomic models and
    $\bbar\subseteq M$, $\abar\subseteq N\setminus M$ such that
    $N\models\theta(\bbar,\abar)$.
    \end{enumerate}}
\end{definition}

Note that an atomic class has an uncountable model if and only if it has a
non-pseudo-algebraic type.

The definition of an extendible formula depends on the partition of its free variables.
 As we require extendible
 formulas to be complete, they are not
preserved under adjunction of dummy variables.  If $\lg(\xbar)=1$, then $\theta(\zbar,x)$ being extendible is equivalent to it being complete, with
$\theta(\zbar,x)$ not pseudo-algebraic.  Much of the utility of the notion is given by the following fact.

\begin{fact}    \label{onesuffices}
\begin{enumerate}
\item  If $\theta(\zbar;\xbar)$ is extendible, then
 for any countable, atomic $M$ and any $\bbar\in M^{\lg(\zbar)}$ and
 $\abar\in M^{\lg(\xbar)}$ such that $M\models\theta(\bbar,\abar)$,
there is $M_0\preceq M$ such that $\bbar\subseteq M_0$ and $\abar\subseteq M\setminus M_0$.
\item  If $\theta(\zbar;\xbar)$ is extendible and $\zbar'\subseteq \zbar$ and
$\xbar'\subseteq \xbar$, then the restriction $\theta\mr{\zbar';\xbar'}$ is extendible as well.
\item  Any complete formula $\theta(\zbar;\xbar)$ is extendible if and only
if $\theta\mr{\zbar,x_i}$ is not pseudo-algebraic for every $x_i\in\xbar$.
\end{enumerate}
\end{fact}

\begin{proof}  (1)  As $\theta(\zbar;\xbar)$ is extendible, choose countable atomic models $M'\preceq N'$, $\bbar'\subseteq M$ and $\abar'\subseteq N'\setminus M'$
such that $N'\models\theta(\bbar',\abar')$.  As $\theta(\zbar;\xbar)$ is complete, there is an isomorphism $f:N'\rightarrow M$ with $f(\bbar')=\bbar$ and $f(\abar')=\abar$.
Then $M_0:=f(M')$ is as desired.

(2)  This follows easily from the proof of (1).

(3)  Left to right follows easily from (2).  We prove the converse by induction on $\lg(\xbar)$.  For $\lg(\xbar)=1$ this is immediate, so assume this holds when $\lg(\xbar)=n$.
Choose a complete $\theta(\zbar;\xbar, x_n)$ such that $\lg(\xbar)=n$ and $\theta\mr{\zbar,x_i}$ is non-pseudoalgebraic for each $i\le n$.  Choose any countable, atomic $N$ and
$\bbar,\abar,a_n$ from $N$ so that $N\models\theta(\bbar,\abar,a_n)$.  By (1), it suffices to find some $M_0\preceq N$ with $\bbar\subseteq M_0$ and $\abar a_n\subseteq N\setminus M_0$.
To obtain this, since $\exists x_n\theta(\zbar;\xbar,x_n)$ is extendible by (2),  (1) implies there is $M\preceq N$ with $\bbar\subseteq M$
and $\abar\subseteq N\setminus M$.  Thus, if $a_n\in N\setminus M$, we can take $M_0:=M$ and we are done.  If not, then
as $\bbar a_n\subseteq M$ we can apply (1) to $M$ and the extendible $\exists \xbar \theta(\zbar;\xbar,x_n)$ to get $M_0\preceq M$ with $\bbar\subseteq M_0$ and $a_n\in M\setminus M_0$.
\end{proof}

Next, we consider the `transitive closure' of extendibility.

\begin{definition}  {\em  An {\em $n$-striated formula} is a complete formula $\theta(\ybar_0,\dots,\ybar_{n-1})$ whose free variables are partitioned into $n$ pieces such
that, for every $i<n$, letting $\zbar=(\ybar_0,\dots,\ybar_i)$ and $\xbar=(\ybar_i,\dots,\ybar_{n-1})$, we have $\theta(\zbar,\xbar)$ extendible.

A striated formula is an $n$-striated formula for some $n$.

A {\em realization} of an $n$-striated formula $\theta(\ybar_0,\dots,\ybar_{n-1})$
is an $n$-chain $M_0\preceq M_1\preceq M_{n-1}$ of countable, atomic models,
together with tuples $\abar_0,\dots,\abar_{n-1}$ with $\abar_0\subseteq M_0$
and $\abar_i\subseteq M_i\setminus M_{i-1}$ for every $0<i<n$ such that $M_{n-1}\models\theta(\abar_0,\dots,\abar_{n-1})$.
}
\end{definition}

Iterating Fact~\ref{onesuffices}, we see that a partitioned complete formula
$\theta(\ybar_0,\dots,\ybar_{n-1})$ is $n$-striated if and only if for some
countable atomic $M$ and some $(\abar_0,\dots,\abar_{n-1})$ from $M$ with
$M\models\theta(\abar_0,\dots,\abar_{n-1})$, there are $M_0\preceq M_1\preceq
\dots\preceq M_{n-2}\preceq M$ with $\abar_0\subseteq M_0$, $\abar_i\subseteq
M_i\setminus M_{i-1}$ for $0<i<n-2$ and $\abar_{n-1}\cap M_{n-2}=\emptyset$.

Using this characterization, if $\theta(\ybar_0,\dots,\ybar_{n-1})$ is $n$-striated and we modify the partition of $\theta$ by fusing together two adjacent tuples, then the resulting partition yields an
$(n-1)$-striated formula.
Going forward, we have the following amalgamation property for striated formulas.

\begin{lemma}  \label{amalgamation}  Suppose $\alpha(\zbar,\xbar_1,\dots,\xbar_n)$ and $\beta(\zbar,\ybar_1,\dots,\ybar_m)$
are striated and $\alpha\mr{\zbar}$ is equivalent to $\beta\mr{\zbar}$.  Then
there is a striated $\psi(\zbar,\xbar_1,\dots,\xbar_n,\ybar_1,\dots,\ybar_m)$
extending
$\alpha(\zbar,\xbar_1,\dots,\xbar_n)\wedge\beta(\zbar,\ybar_1,\dots,\ybar_m)$.
\end{lemma}

\begin{proof}  Choose an $(n+1)$-chain $M_0\preceq M_1\preceq\dots\preceq M_n$
and $\bbar,\abar_1,\dots,\abar_n$  realizing $\alpha$ (so $\bbar\subseteq M_0$
 and $\abar_i\subseteq M_i\setminus M_{i-1}$ for each $i$) and choose similarly an
  $(m+1)$-chain $N_0\preceq N_1\preceq \dots\preceq N_m$ and $\cbar,\dbar_1,\dots,\dbar_m$ realizing $\beta$.
As $\alpha\mr{\zbar}$ is equivalent to $\beta\mr{\zbar}$, there is an isomorphism $f:N_0\rightarrow M_n$ with $f(\cbar)=\bbar$.  Choose $M_{n+m}\succeq M_n$ for which there is an
isomorphism $f^*:N_m\rightarrow M_{n+m}$ extending $f$.  Now, for $i\le m$ put $M_{n+i}:=f^*(N_i)$.  [Note this is compatible with our previous placements.]  Also,
for each $1\le i\le m$, put $\abar_{n+i}:=f^*(\dbar_i)$.  Finally, put
$\psi(\zbar,\xbar_1,\dots,\xbar_n,\ybar_1,\dots,\ybar_m):=\tp(\bbar,\abar_1,\dots,\abar_{n+m})$.
Then the $(n+m+1)$-chain $M_0\preceq\dots\preceq M_{n+m}$, together with $\bbar,\abar_1,\dots,\abar_{n+m}$ witness that $\psi$ is striated.
\end{proof}

\subsection{The forcing} \label{theforcing} We continue our assumption that we have a fixed complete theory $T$ in a countable language with an uncountable atomic model.
We fix an $\aleph_1$-like dense linear order $(I,\le)$ with least element $0$
and fix a continuous, increasing (necessarily cofinal) sequence
$\<J_\alpha:\alpha\in\omega_1\>$ of initial segments of $I$. Also, fix a set
$X=\{x_{t,m}:t\in I,m\in\omega\}$ of distinct variable symbols and, for each
$\alpha\in\omega_1$, let $X_\alpha=\{x_{t,m}:t\in J_\alpha,m\in\omega\}$. Our
forcing below will describe a complete diagram in the variables $X$
corresponding to an atomic model $N$ of size $\aleph_1$ and the countable
substructures $N_\alpha$ corresponding to the variables $X_\alpha$ will be a
filtration of $N$.

\begin{definition}\label{forcedef}  {\em
The forcing $(\P,\le)$ consists of all conditions
$$p=(u_p,\ell(p),\{k_{p,i}:i<\ell(p)\},\theta_p(\ybar_0,\dots,\ybar_{\ell(p)-1)}))$$
satisfying the following properties:
\begin{enumerate}
\item  $u_p$ is a finite subset $\{s_0,\dots,s_{\ell(p)-1}\}\subseteq I$.  We always write the elements of $u_p$ in ascending order.
\item  $\ell(p)=|u_p|$;
\item  If $u_p\neq\emptyset$, then $0\in  u_p$;
\item  Each $k_{p,i}\in\omega$ and denotes $\lg(\ybar_i)$ in $\theta_p$;
\item  $\theta_p(\ybar_0,\dots,\ybar_{\ell(p)-1}))$ is an
    $\ell(p)$-striated formula, where each
    $\ybar_i=(\xbar_{s_i,j}:j<k_{p,i})$ is the initial segment of the
    $s_i$'th row of $X$ of length $k_{p,i}$.
\end{enumerate}}
\end{definition}

%
%
%

The ordering on $\P$ is natural, i.e., $p\le_{\P} q$ if and only if
 $u_p\subseteq u_q$, the free variables of $\theta_p$
        are contained in the free variables of $\theta_q$  and
        $\theta_q\vdash \theta_p$.


%

We remark that the effect of requiring $0\in u_p$ whenever $u_p$ is non-empty is to ensure that if $\theta_p$ entails `$x_{\alpha_i,j}\in\pcl(\emptyset)$', then $\alpha_i=0$.
That is, in the generic model we construct, all pseudo-algebraic complete types of singletons will be contained in $M_0$.

It is easily verified that $(\P,\le)$ is c.c.c.~(See \cite[Claim 4.3.7]{BLS}
for a verification of this in an extremely similar setting.) We record three
additional density conditions about $(\P,\le)$
whose verifications depend on the following fact.

%


\begin{lemma}  \label{nonpseudo}  Suppose $\delta(x)$ is a non-pseudoalgebraic 1-type.
  Then for every countable atomic $N$ and every $\ebar\subseteq N$,
there are $M\preceq N$ and $c\in N\setminus M$ such that $\ebar\subseteq M$ and $N\models \delta(c)$.
\end{lemma}

\begin{proof}  From the definition of (non)-pseudoalgebraicity, fix countable atomic $M^*\preceq N^*$ and $c^*\in N^*\setminus M^*$ with $N^*\models \delta(c^*)$.
Choose any isomorphism $f:N\rightarrow M^*$ and put $\ebar^*:=f(\ebar)$.  Now, choose an isomorphism $g:N^*\rightarrow N$ with $g(\ebar^*)=\ebar$.
Put $M:=g(M^*)$ and $c:=g(c^*)$.  Then $\ebar\subseteq M$, $c\in N\setminus M$, and $N\models \delta(c)$.
\end{proof}
%

The forcing is surjective in the sense that for every condition $p$ and every
variable there is an extension of $p$ that includes the variable.

\begin{lemma}[{\bf Surjective}]  For every $p\in \P$ and $x_{t,m}\in X$, there is $q\in\P$, $q\ge p$ with $\xbar_q=\xbar_p\cup\{x_{t,m}\}$.
\end{lemma}

\begin{proof}  We may assume that $p\neq 0$ and that $x_{t,m}\not\in \xbar_p$.  Choose $M_0\preceq M_1\preceq\dots\preceq M_{n-1}$ and $\ebar_0\dots\ebar_{n-1}$
realizing $\theta_p$ (so $\ebar_0\subseteq M_0$, $\ebar_i\subseteq M_i\setminus M_{i-1}$ for $0<i<n$ and $M_{n-1}\models\theta(\ebar_0,\dots,\ebar_{n-1})$.

We first handle the case where $m=0$.  In this case, it must be that $t\not\in u_p$.  Choose $j$ maximal such that $s_j<t$.
Apply Fact~\ref{nonpseudo} to $M_{j}$ and $\ebar_0\dots\ebar_j$ to get
$M_j^*\preceq M_{j}$ and $c\in M_{j}\setminus M_j^*$ with
$M_{j}\models\delta(c)$ and $\ebar_0\dots,\ebar_{j}\subseteq M_j^*$.  Now let
$f:M_j\rightarrow M_j^*$ be an isomorphism fixing $\ebar_0\dots,\ebar_{j}$
pointwise. Then  the type $\tp(\ebar_0,\dots,\ebar_j,c,\ebar_{j+1},\dots,
\ebar_{n-1})$ and the $(n+1)$-chain $f(M_0)\preceq \dots f(M_{j})\preceq
M_j\preceq M_{j+1}\preceq\dots M_{n-1}$ describes an $(n+1)$-striated formula
$\theta$.  Let $q\in\P$ be the element with $\xbar_q=\xbar_p\cup\{x_{t,0}\}$
with $\theta_q(\xbar_q)$ being the complete formula generating this type.

If $m>0$, then we apply the previous case to ensure that $x_{t,0}\in\xbar_p$.  Say $t=s_j$, the $j$'th element of $u_p$.
But then, given any $\ebar_0,\dots,\ebar_{n-1}$ and $M_0\preceq\dots\preceq M_{n-1}$
realizing $\theta_p$, extend $\xbar_{p,t}$ to include $x_{t,m}$ by making each `new' element of $\ebar_j$ equal to the element $e_{j,0}\in M_j$.
\end{proof}

The notational issue in what follows is the placement of free variables,  For $p\in \P$, there is an explicit ordering to the variables $\xbar_p$ occurring in
$\theta(\xbar_p)$, but when we consider
extensions $\phi(\vbar,\xbar_p)$, we do not want to specify where the $v_i$'s fit in the sequence.

\begin{lemma}[{\bf Henkin}]  Suppose $p\in\P$ and $\theta_p(\xbar_p)\vdash \exists\vbar \phi(\vbar,\xbar_p)$.  Then there is $q\in\P$, $q\ge p$
for which the variables in $(\xbar_q\setminus \xbar_p)$ consist of a realization of $\phi(\vbar,\xbar_p)$ (in some order).  Moreover, if $p\neq 0$, then can be chosen with
$u_q=u_p$.
\end{lemma}

\begin{proof}  Arguing by induction, we may assume $\vbar=\{v\}$ is a singleton,
 and we may further assume that $\phi(v,\xbar_p)$ describes a complete type.
 Let $\ebar_0,\dots,\ebar_{n-1}$ and $M_0\preceq\dots\preceq
M_{n-1}$  witness the truth and striation of $\theta_p$ and choose any $b\in
M_{n-1}$ such that
 $M_{n-1}\models\phi(b,\ebar_p)$.
Let $j\le n-1$ be least such that $b\in M_{j}$.  (Note that if
 $\phi(v,\xbar_p)\vdash$ `$v\in\pcl(\emptyset)'$, then we must have $j=0$.)
Let $\xbar_q=\xbar_p\cup\{x_{s_j,k_p(j)}\}$.  Then, letting
$\ebar_j^*=\ebar_j b$, we have a striation
$\ebar_0,\dots,\ebar_{j-1},\ebar_j^*,\ebar_{j+1},\dots,\ebar_{n-1}$ using the
same $n$-chain of models $M_0\preceq\dots M_{n-1}$. Put
$$\theta_q(\xbar_q):=\tp(\ebar_0,\dots,\ebar_{j-1},\ebar_j^*,\ebar_{j+1},\dots,\ebar_{n-1}).$$
 Then $q\in\P$ and $q\ge p$.
\end{proof}

\begin{lemma}  \label{amalg}  Suppose $p,q,r\in\P$ with $p\le q$, $p\le r$, $\xbar_q\cap\xbar_r=\xbar_p$, and for some $t\in I$, $u_q\subseteq I_{<t}$ and $(u_r\setminus u_p)\subseteq I_{>t}$.
Suppose further that there are $M\preceq N$ and $\abar,\bbar,\cbar$ with $\bbar\cap\cbar=\abar$, $\bbar\subseteq M$, and $(\cbar\setminus\abar)\subseteq N\setminus M$
with $N\models\theta_p(\abar)\wedge\theta_q(\bbar)\wedge\theta_r(\cbar)$.
Then there is $r^*\in\P$, $r^*\ge q$, $r^*\ge r$ with $\xbar_{r^*}=\xbar_q\cup\xbar_r$ and $\theta_{r^*}=\tp(\bbar,(\cbar\setminus\abar))$.
\end{lemma}

\begin{proof}  Arguing by induction, we may additionally assume that $u_r=u_p\cup\{s^*\}$
for some single $s^*>t$.  That is, $\xbar_q\setminus\xbar_p$ lies on a single
level of $X$. Since $q\in\P$, there is a striation of
$\bbar=\bbar_0,\bbar_1,\dots,\bbar_{n-1}$ induced by the rows of $\xbar_q$.
As $\bbar\subseteq M$, we can find an $n$-chain $M_0\preceq M_{n-1}$ of
models with $M_{n-1}=M$, $\bbar_0\subseteq M_0$ and $\bbar_i\subseteq
M_i\subseteq M_{i-1}$ for all $0<i<n$. As $(\cbar\setminus\abar)\subseteq
N\setminus M$ and as $(\xbar_q\setminus \xbar_p)$ consists of a single row
(and since $s^*>t$) it follows that the $(n+1)$-tuple
$\bbar_0,\dots,\bbar_{n-1},(\cbar\setminus\abar)$ is realized in the
$(n+1)$-chain $M_0\preceq\dots\preceq M\preceq N$. Choose
$\xbar_{r^*}=\xbar_q\cup\xbar_r$ and put
$\theta_{r^*}=\tp(\bbar_0,\dots,\bbar_{n-1},(\cbar\setminus\abar))$.  Then
$r^*\in\P$ and both $r^*\ge q$, $r^*\ge r$ hold.
\end{proof}

Armed with these lemmas, we can now prove the main fact about the forcing
$(\P,\le)$ and the generic model $\NN$ of $T$. For forcing notation see
\cite{Kunen}.

\begin{notation}   {\em In what follows, when dealing with $L$-formulas, we will use the letters $\ubar,\vbar,\wbar$, possibly with decorations to denote free variables.
By contrast, tuples denoted by $\xbar,\ybar,\zbar$ denote finite tuples from $X$.  Thus, for example, $\eta(\vbar,\zbar)$ has free variables $\vbar$, and $\zbar$ is a fixed
tuple from $X$.
}
\end{notation}

\begin{remark}  \label{explain}  {\em  In what follows we use the following consequence of splitting inside an atomic model.
Suppose $M\preceq N$ are atomic, $\abar\in M$, $\bbar\in N$, and
$\tp(\bbar/M)$ splits over $\abar$. Then, letting $\theta(\ubar)$ isolate the
complete type of $\abar$ and $\theta'(\wbar,\ubar)$ isolate the complete type
of $\bbar\abar$, there must be a complete formula $\eta(\vbar,\ubar)\vdash
\theta(\ubar)$ and two contradictory complete formulas
$\delta_1(\wbar,\vbar,\ubar)$ and $\delta_2(\wbar,\vbar,\ubar)$, each
extending the (incomplete) formula
$\eta(\vbar,\ubar)\wedge\theta'(\wbar,\ubar)$. }
\end{remark}

\begin{proposition}  \label{forcingresult}  $(\P,\le)$ forces $\NN$ is an atomic model of size $\aleph_1$ with every limit type in $S_{at}(\NN)$ is constrained.
\end{proposition}
\begin{proof}
 It follows from the Surjective (atomic) and Henkin density conditions (model) that $\NN$ is an atomic model of $T$ with size $\aleph_1$.  Moreover, for each $\alpha\in \omega_1$,
$\NN_\alpha$, the substructures $\NN_\alpha$ indexed by the subsets
$X_\alpha$ form a filtration of $\NN$.

Call   a function $b:\omega_1\rightarrow \NN$ a {\em limit sequence} if,
 for all $\alpha\le\beta$,
$\tp(b(\alpha)/\NN_\alpha)=\tp(b(\beta)/\NN_\alpha)$.
Now, if  $(\P,\le)$ does not force that every limit type is constrained,  then there is some
$p^*\in\P$
and  some $\P$-name $\bb$ and some club $C\subseteq \omega_1$ such that
$$p^*\forces \bb\ \hbox{is a limit sequence with $\tp(\bb(\alpha)/\NN_\alpha)$ unconstrained for every $\alpha\in C$}$$
(Since $(\P,\le)$ is c.c.c.\ we can find such a club in $\VV$.)

%
 For each
$\alpha\in C$, choose $p_\alpha\in\P$, $p_\alpha\ge p^*$ and $x^*_\alpha\in
X$ such that
$$p_\alpha\forces \bb(\alpha)=x^*_\alpha$$
We will eventually reach a contradiction by finding some $q^*\ge p^*$ and some $\alpha<\beta$ from $C$ such that
$$q^*\forces \tp(x_\alpha^*/N_\alpha)\neq\tp(x_\beta^*/N_\alpha)$$

By a routine $\Delta$-system argument, find a `root' $p_0\in \P$, some $\gamma^*\in\omega_1$, and
a stationary set $S\subseteq C$
 satisfying:
\begin{itemize}
\item $p_0\le p_\alpha$ for all $\alpha\in S$;
\item  $u_{p_0}\subseteq J_{\gamma^*}$;
and
\item for all $\alpha<\beta$ in $S$,
\begin{itemize}
\item $\xbar_{p_\alpha}\cap X_{\gamma^*}=\xbar_{p_0}$;
\item  $\max(u_{p_\alpha})<\min(u_{p_\beta}\setminus u_{p_0})$;
\item  $\lg(p_\alpha)=\lg(p_\beta)$ and $k_{p_\alpha}=k_{p_\beta}$; and
\item  The formulas $\theta_{p_\alpha}$ and $\theta_{p_\beta}$ have the
    same syntactic shape [one formula can be obtained from the other by
    substituting the free variables].
\end{itemize}
\end{itemize}

Note that we do not require $p_0\ge p^*$.   As notation, we write
$\zbar$ for $\xbar_{p_0}$ and note that $\zbar\subseteq X_{\gamma^*}$.
Now fix, for the remainder of the argument,
 some $\alpha<\beta$ from $S$.  To obtain our desired contradiction, we
first concentrate on $p_\alpha$.  Write $\theta_{p_\alpha}(\ybar,\zbar)$ and note that $\ybar$ is disjoint from $X_{\gamma^*}$.
We apply Remark~\ref{explain}, noting that $p_{\alpha}\forces \tp(x_\alpha^*/\NN_\alpha)$ splits over $\zbar$.
Choose a complete formula $\eta(\vbar,\zbar)$ implying $\theta_{p_0}(\zbar)$ and contradictory complete formulas
$\delta_1(x^*_\alpha,\vbar,\zbar)$ and $\delta_2(x^*_\alpha,\vbar,\zbar)$, each extending $\eta(\vbar,\zbar)\wedge\theta^*_{p_\alpha}(x^*_\alpha,\zbar)$,
where $\theta^*_{p_\alpha}$ is the restriction of the compete formula $\theta_{p_\alpha}(\ybar,\zbar)$.

By Henkin, choose $q_0\in \P$, $q_0\ge p_0$ with
$u_{q_0}\subseteq J_\alpha$ and
$\theta_{q_0}(\zbar',\zbar):=\eta(\zbar',\zbar)$.
Next, we use Lemma~\ref{amalg} twice.  In both cases we start with $p_0\le
q_0$ and $p_0\le p_\alpha$.  Our first application gives $r^1_\alpha\in \P$
extending both $q_0$ and $p_\alpha$ with
$\theta_{r_\alpha^1}(\ybar,\zbar',\zbar)\vdash\delta_1(x_\alpha^*,\zbar',\zbar)$.  The second application
gives $r^2_\alpha\in \P$, also extending both $q_0$ and $p_\alpha$ with
$\theta_{r_\alpha^2}(\ybar,\zbar',\zbar)\vdash\delta_2(x_\alpha^*,\zbar',\zbar)$.

Next, we use the fact that the forcing $(\P,\le)$ is highly homogeneous.  Due to
 the similarity of $p_\alpha$ and $p_\beta$ found by the $\Delta$-system argument,
  there is an automorphism $\sigma$ of $(\P,\le)$ sending
$p_\alpha$ to $p_\beta$, fixing $q_0$.  Put $r^2_\beta:=\sigma(r^2_\alpha)$.
 We now apply Lemma~\ref{amalgamation} to $q_0\le r^1_\alpha$ and $q_0\le r^2_\beta$ to get $q^*\in\P$ with $q^*\ge r^1_\alpha$ and $q^*\ge r^2_\beta$.
 However, this is impossible, as
 $$q^*\forces \delta_1(x^*_\alpha,\zbar',\zbar)\wedge \delta_2(x^*_\beta,\zbar',\zbar)$$
 contradicting $p^*\forces \tp(x_\alpha^*/\NN_\alpha)=\tp(x_\beta^*/\NN_\alpha)$ since $\delta_1$ and $\delta_2$ were chosen to be contradictory.
\end{proof}

\subsection{Proof of Theorem~\ref{constrainedexists}}\label{climax}
Theorem~\ref{constrainedexists} follows immediately from the two previous
results and Keisler's model existence result for $L_{\omega_1,\omega}(Q)$. In
particular, by Proposition~\ref{forcingresult}, there is an uncountable atomic   model
with every limit type constrained
in some c.c.c.\ forcing extension $\VV[G]$.  Hence, by
Proposition~\ref{main}, $\VV[G]$ thinks there is a model of $\Psi^*$.  Hence,
by the absoluteness of existence from Keisler's theorem, there is an
uncountable  model of $\Psi^*$ with cofinality $\aleph_1$ in $\VV$. Reversing
the implication in the last sentence of Proposition~\ref{main}, a  second
application gives the existence of such a model in $\V$.

\end{document}